\documentclass[a4paper, reqno, 11pt]{amsart}

\usepackage{amssymb}
\usepackage{hyperref}
\usepackage{enumerate}
\numberwithin{equation}{section}
\usepackage[margin=1.2in]{geometry}

\newtheorem{theorem}{Theorem}[section]
\newtheorem{lemma}[theorem]{Lemma}
\newtheorem{corollary}[theorem]{Corollary}
\newtheorem{proposition}[theorem]{Proposition}
\newtheorem{claim}[theorem]{Claim}
\newtheorem{assumption}{Assumption}

\theoremstyle{definition}
\newtheorem{definition}[theorem]{Definition}
\newtheorem{example}[theorem]{Example}

\theoremstyle{remark}
\newtheorem{remark}[theorem]{Remark}

\allowdisplaybreaks

\begin{document}

% ===========================
% ============= Our macro
\newcommand{\gn}[1]{{
\left\vert\kern-0.25ex
\left\vert\kern-0.25ex
\left\vert
#1 
\right\vert\kern-0.25ex
\right\vert\kern-0.25ex
\right\vert
}}
% ===========================

\title[Yosida Distance and Existence of Invariant Manifolds]{Yosida Distance and Existence of Invariant Manifolds in the Infinite-Dimensional Dynamical Systems}
%----------------------
%----------------------
\author[X.-Q. Bui]{Xuan-Quang Bui}
\address{Faculty of Fundamental Sciences, PHENIKAA University, Hanoi 12116, Vietnam}
\email[corresponding author]{quang.buixuan@phenikaa-uni.edu.vn}
 
\author[N.V. Minh]{Nguyen Van Minh}
\address{Department of Mathematics and Statistics, University of Arkansas at Little Rock, 2801 S University Ave, Little Rock, AR 72204, USA}
\email{mvnguyen1@ualr.edu}
%----------------------
%----------------------
%----------------------
%----------------------
%----------------------
%----------------------
%\thanks{The work of the first author is partly supported by the Vietnam Institute for Advanced Study in Mathematics (VIASM)}
%----------------------
%----------------------
\begin{abstract}
We introduce a new concept of \textit{Yosida distance} between two (unbounded) linear operators $A$ and $B$ in a Banach space $\mathbb{X}$ defined as $d_Y(A,B):=\limsup_{\mu\to +\infty} \| A_\mu-B_\mu\|$, where $A_\mu$ and $B_\mu$ are the Yosida approximations of $A$ and $B$, respectively, and then study the persistence of evolution equations under small Yosida perturbation. This new concept of distance is also used to define the continuity of the proto-derivative of the operator $F$ in the equation $u'(t)=Fu(t)$, where $F \colon D(F)\subset \mathbb{X} \rightarrow \mathbb{X}$ is a nonlinear operator. We show that the above-mentioned equation has local stable and unstable invariant manifolds near an exponentially dichotomous equilibrium if the proto-derivative of $F$ is continuous. The Yosida distance approach to perturbation theory allows us to free the requirement on the domains of the perturbation operators. Finally, the obtained results seem to be new.
\end{abstract}
%----------------------
%----------------------
\date{\today}
%----------------------
%----------------------
\maketitle
%----------------------
%----------------------
\section{Introduction}
It is well known in the qualitative theory of ordinary differential equations that the asymptotic behavior of linear equations of the form
\begin{equation}\label{ode}
u'(t)=Mu(t),\qquad u(t)\in \mathbb{R}^n,\quad t\in \mathbb{R},
\end{equation}
where $M$ is a $n\times n$-matrix, persists under ``small'' perturbation that is either linear or nonlinear if the system has an exponential dichotomy, that is, all eigenvalues of $M$ are off the imaginary axis (see e.g. \cite{cop, dalkre, engnag}), or equivalently, the unit circle does not intersect with the spectrum of $e^M$. For many decades, extensions of these classical results have been ones of the central topics in the theory of infinite dimensional dynamical systems with applications to partial differential equations and other types of evolution equations. Namely, for linear perturbation of a linear hyperbolic system in a Banach space
\begin{equation}\label{eq}
u'(t)=Au(t),\quad u(t)\in \mathbb{X},
\end{equation}
where $\mathbb{X}$ is a (complex) Banach space, $A \colon D(A)\subset \mathbb{X} \to \mathbb{X}$ is an unbounded linear operator, one often considers the equation
\begin{equation}\label{eql}
u'(t)=(A+B)u(t),\quad u(t)\in \mathbb{X},
\end{equation}
where $B \colon D(B)\subset \mathbb{X}\to\mathbb{X}$ is a linear operator. To justify for the ``smallness'' of the perturbation $B$ one often assumes that $B$ is bounded and its norm $\|B\|$ is small. This assumption on the ``smallness'' of perturbation $B$ actually limits its applicability of the obtained results to partial differential equations and other types of evolution equations. For this reason, when $A$ has some further properties like the generator of an analytic semigroup one can consider the perturbation $B$ among a more general class of linear operator that are $A$-bounded, that is, $\|Bx\| \le a\|Ax\| +c\|x\|$, $x\in D(A)$, for some fixed positive constants $a$ and $c$. Then, the smallness of $B$ is measured by the sizes of $\max \{a,c\}$. Of course, in these cass a requirement that $D(A) \subset D(A+B)$ is indispensable. These approaches are discussed in \cite{cholei, dalkre, engnag, paz, pru} for the generation of semigroups by $A+B$ that can be easily used to show that the exponential dichotomy of Eq. \eqref{eql} persists under small perturbation.

For nonlinear perturbation of Eq. \eqref{eq}
\begin{equation}\label{eqn}
u'(t)=Au+Fu,\quad u(t)\in \mathbb{X},
\end{equation}
where $F:D(F)\subset \mathbb{X}\to\mathbb{X}$ is a nonlinear operator, the asymptotic behavior of Eq. \eqref{eqn} near the equilibrium is often described by the (local) invariant manifolds near the equilibrium, see e.g. \cite{batluzen, batluzen2, cholu, pralun, heipru, lun, lun95, minwu, tem}. Apart from the assumption that $F$ be differentiable in some sense, say, in the sense of Fr\'{e}chet, the continuity of the derivatives is often determined by that of $F'(x)$ in certain spaces of bounded linear operators (see e.g. \cite{pralun, heipru, lun, lun95}). This allows to include more classes of partial differential equations into the consideration.

In this paper we will take an attempt to propose a new approach to the perturbation theory for Eq. \eqref{eq} in which the size of perturbation will be measured by the so-called ``Yosida distance''. This is a new concept of a pseudo metric defined on the set of closed operators, especially, the set of all generators of $C_0$-semigroups (see Definition~\ref{def 1} below). This allows us to measure the distance between two unbounded operators in many important classes as seen in our Lemma~\ref{lem 1} and Example \ref{exa 1}. Note that in our Examples \ref{exa 1} and ~\ref{exa age} Yosida distances between two unbounded operators could be determined while the domain of one operator may not contain that of the other. This strict requirement is often seen in previous works on perturbation theory of evolution equations and $C_0$-semigroups (see, for instance, \cite{cholei, dunsch, lun, paz}). We will use the Yosida distance to define the continuity of the proto-derivatives of $F$ in Eq. \eqref{eqn}. For linear equations \eqref{eql} we will show the persistence of the exponential dichotomy under small perturbation measured by Yosida distance between $A$ and $A+B$, see Theorem~\ref{the per}. The Yosida distance approach will be extended to nonlinear perturbation, namely, to study Eq. \eqref{eqn} where we allow $F$ to be proto-differentiable and its proto-derivative is continuous in the topology defined by the Yosida distance. We prove the existence of local stable and unstable invariant manifolds near an equilibrium of Eq. \eqref{eqn} if the linearized equation \eqref{eql} has an exponential dichotomy, see Theorems \ref{TM-Thm-USlocal} and \ref{TM-Thm-SMlocal}. We note that in our Assumptions \ref{M-assump124} and \ref{M-assump123} conditions on the m-accretiveness of the operators only to guarantee that they generate semiflows, or in other words, the corresponding equations are well-posed. At the end of the paper we provide some examples to show that our approach allows a generalization of known results as shown in Examples \ref{exa 5.1}, \ref{exa 5.2} and \ref{exa age}. To the best of our knowledge, the Yosida distance approach to the perturbation of exponential dichotomy and the obtained results discussed in this paper are new.

This paper is organized as follows:
In Section~\ref{Sect-Pr}, we first list some notations used in the paper. Then, we recall some background materials on accretive operators and generation of strongly continuous nonlinear semigroups.
Section~\ref{Sect-MainResults} contains the main results with the concept of Yosida distance between two linear operators and the roughness of exponential dichotomy in linear dynamical systems.
In Subsection~\ref{Subsect-InvMa}, using the concept of Yosida distance, we define the continuity of the proto-derivative of a nonlinear operator and prove the existence of invariant manifolds.
Finally, in Section~\ref{Sect-AE}, we give some examples to show that the Yosida distance is a consistent tool in studying the existence of invariant manifolds.

\section{Preliminaries}\label{Sect-Pr}
\subsection{Notations}
In this paper we will denote by $\mathbb{X},\, \mathbb{Y}$ Banach spaces with corresponding norms. The symbols $\mathbb{R}$ and $\mathbb{C}$ stand for the fields of real and complex numbers, respectively. Denote by $\mathcal{L}(\mathbb{X},\mathbb{Y})$ the Banach space of all bounded linear operators from a Banach space $\mathbb{X}$ to a Banach space $\mathbb{Y}$. We will denote the domain of an operator $T$ by $D(T)$ and its range by $R(T)$. The resolvent set of a linear operator $T$ in a Banach space will be denoted by $\rho(T)$ while its spectrum is denoted by $\sigma (T)$.
For $\lambda \in \rho (T)$, we denote the inverse $(\lambda I - T)^{-1}$ by $R(\lambda, T)$ and call it the resolvent (at $\lambda$) of linear operator $T$.
Let $P$ be a linear operator in a Banach space $\mathbb{X}$, as usual, $\mathrm{Ker}(P)$ and $\mathrm{Im}(P)$ are the notations of kernel and image of $P$.
We shall denote by $B_r (0, \mathbb{Y})$ the ball of radius $r$ centered at $0$ of $\mathbb{Y}$ and write $B_r (0, \mathbb{X}) = B_r (0)$ if this does not cause any confusion, when $\mathbb{X}$ is a given fixed Banach space.
Given function $F$, the notations $F'(u)$, $dF(u)$, and $\partial F(u)$ will stand for the 
Fr\'{e}chet derivative,
G\^{a}teaux derivative, and
proto-derivative, respectively, of $F$ at $u$.

\subsection{Accretive Operators and Generation of Strongly Continuous Nonlinear Semigroups}
\begin{definition}[accretive, m-accretive (see \cite{bar, cra})]\label{def accretive}
Let $B$ be a (possibly nonlinear multi-valued) operator in $\mathbb{X}$, then $B$ is called \textit{accretive} if $(I+\lambda B)^{-1}$ exists as a single-valued function and 
$$
\left \| (I+\lambda B)^{-1}x-(I+\lambda B)^{-1}y\| \le  \| x-y\right \|,
$$ 
for all $x,\,y \in D((I+\lambda B)^{-1})$.
An accretive operator $B$ is called \textit{m-accretive} if $R(I+\lambda B)=\mathbb{X}$ for all (equivalently for some) $\lambda >0$.
\end{definition}

Below is the well known Crandall-Liggett Theorem on the generation of strongly continuous nonlinear semigroups.
\begin{theorem}[see Crandall-Liggett \cite{cralig}]\label{the cra}
Let $A$ be a (possibly multivalued) nonlinear operator and $\omega$ be a real number such that $\omega I-A$ is accretive. If $R(I-\lambda A) \supset \overline {D(A)}$ for all sufficiently small positive $\lambda$, then
\begin{equation}\label{M-CLthm}
\lim_{n\to\infty} \left( I-\frac{t}{n}A\right)^{-1} x
\end{equation}
exists for all $x\in \overline{D(A)}$ and $t>0$. Moreover, if $S(t)$x  is defined as the limit in \eqref{M-CLthm}, then 
\begin{align}
& S(t+\tau) = S(t)S(\tau),\quad t,\, \tau \ge 0,
\\
& \lim_{t\downarrow 0} S(t)x = x, \quad x\in \overline{D(A)},
\\
& \| S(t)\| \le e^{\omega t}, \quad t\ge 0.
\end{align}
\end{theorem}
In addition, if $A$ is a single-valued operator and $R(I-\lambda A)\supset \mathrm{clco} \ D(A)$ (``$\mathrm{clco}$'' means the closure of convex hull of $D(A)$), then $S(t)x$ is the solution of the Cauchy problem
\begin{equation}
\frac{du}{dt}=Au,\quad u(0)=x \in D(A) .
\end{equation}

\subsection{Exponential Dichotomy}
\begin{definition}[Exponential dichotomy]\label{TM-Definition1.2}
A linear semigroup $(T(t))_{t \ge 0}$ is said to have an \textit{exponential dichotomy} or to be \textit{hyperbolic} if there exist a bounded projection $P$ on $\mathbb{X}$ and positive constants $N$ and $\alpha$ satisfying
\begin{enumerate}
\item $T(t) P = P T(t)$, for $t \ge 0$;
\item $T(t)\big|_{\mathrm{Ker}(P)}$  is an isomorphism from $\mathrm{Ker}(P)$ onto $\mathrm{Ker}(P)$, for all $t \ge 0$, 
and its inverse on $\mathrm{Ker}(P)$ is defined by $T(-t):=\left (T(t)\big|_{\mathrm{Ker}(P)}\right )^{-1}$;
\item the following estimates hold 
\begin{align}
&
\|T(t)x\| \le N e^{-\beta t} \|x\|,
\quad \mbox{ for all }
t \ge 0,\quad
x \in \mathrm{Im}(P),
\\
&
\|T(-t) x\| \le N e^{-\beta t} \|x\|,
\quad \mbox{ for all }
t \ge 0,\quad
x \in \mathrm{Ker}(P).
\end{align}
\end{enumerate}
\end{definition}

The projection $P$ is called the \textit{dichotomy projection} for the hyperbolic semigroup $(T(t))_{t \ge 0}$, and the constants $N$ and $\alpha$ are called \textit{dichotomy constants}.

The following result is well known:
\begin{lemma}\label{Lemhyp}
Let $(T(t))_{t \ge 0}$ be a $C_0$-semigroup. Then, it has an exponential dichotomy if and only if $\sigma (T(1)) \cap \{ z\in \mathbb{C}: |z|=1\}=\emptyset$.
\end{lemma}
\begin{proof}
See Engel-Nagel \cite[1.17 Theorem]{engnag}.
\end{proof}

From Lemma~\ref{Lemhyp}, it is easy to prove the following result:
\begin{lemma}
Let $(T(t))_{t \ge 0}$ be a $C_0$-semigroup that has an exponential dichotomy. Then, $(S(t))_{t \ge 0}$ has an exponential dichotomy provided that $S(1)$ is sufficiently close to $T(1)$.
\end{lemma}

\section{Main Results}\label{Sect-MainResults}

\subsection{Yosida Distance}\label{Subsect-YD}
We begin this section with the concept of Yosida distance between two linear operators. To this end, we recall the concept of Yosida approximation of a linear operator in the definition below (see e.g. \cite{paz,yos}).

Given operator $A$ in a Banach space $\mathbb{X}$ with $\rho(A) \supset [\omega, \infty)$, where $\omega$ is a given number, the \textit{Yosida approximation} $A_\lambda$ is defined as $ A_\lambda := \lambda^2 R(\lambda, A)-\lambda I$
for sufficiently large $\lambda$.

\begin{definition}[Yosida distance]\label{def 1}
The \textit{Yosida distance} between two linear operators $A$ and $B$ satisfying $\rho(A) \supset [\omega, \infty)$ and $ \rho(B) \supset [\omega, \infty)$, where $\omega$ is a given number, is defined to be
\begin{equation}
d_{Y}(A,B):=\limsup_{\mu\to +\infty} \left \| A_\mu-B_\mu\right \| .
\end{equation}
\end{definition}

\begin{lemma}\label{lem 1}
The following assertions are valid:
\begin{enumerate}[\rm (i)]
\item\label{lem1-1} Let $A,\,B$ be the generators of contraction semigroups. Assume further that $D(A)= D(B)$. Then, $A=B$, provided that 
$d_{Y}(A,B)=0$.
\item\label{lem1-2} Let $A,\, B\in \mathcal{L}(\mathbb{X})$. Then
\begin{equation}
d_{Y}(A,B)=\| A-B\|.
\end{equation}
\item\label{lem1-3} If $A$ is the generator of a $C_0$-semigroup $T(t)$ such that $\| T(t)\| \le Me^{\omega t}$, and $C$ is a bounded operator, then $d_{Y}(A,A+C)$ is finite. Moreover,
\begin{equation}
d_{Y}(A,A+C)\le M^2\|C\|
\end{equation}
\item\label{lem1-4} Let $A$ be the generator of an analytic semigroup $(T(t))_{t\ge 0}$ such that $\|T(t)\| \le Me^{\omega t}$, and $C$ be an $A$-bounded operator, that is, $D(C) \supset D(A)$, and there are positive constants $a,\,c$ such that
\begin{equation}\label{A-bounded}
\| Cx \| \le a\| Ax\| + c\| x\| .
\end{equation}
Then, there exists a constant $\delta >0$ such that if $0\le a\le \delta$, the distance $d_{Y}(A,A+C)$ is finite. Moreover,
\begin{equation}\label{ana 4}
d_{Y}(A,A+C) \le aKM+cM^2,
\end{equation}
where $K$ and $M$ are positive constants that depend only on $A$.
\end{enumerate}
\end{lemma}
\begin{proof}
\
\begin{enumerate}[\rm (i)]
\item
From the semigroup theory 
(see Pazy~\cite[Lemma 1.3.3]{paz}),
if $A$ is the generator of a contraction $C_0$-semigroup, then
\begin{equation}
\lim_{\mu\to+\infty} A_\mu x=Ax,
\quad
x\in D(A).
\end{equation}
Therefore, for $x\in D(A)=D(B)$, $Ax=Bx$. 

\item Firstly, we have
\begin{align*}
R(\mu, A) - R(\mu, B)
=
(\mu - B)
R(\mu, B)
R(\mu, A)
-
R(\mu,B)
R(\mu,  A)
(\mu - A).
\end{align*}
Set
$
C_{\mu}
:=
R(\mu, B)
R(\mu, A) 
$.
Then
\begin{align*}
R(\mu, A) - R(\mu, B)
&= 
(\mu - B)
C_{\mu}
-
C_{\mu}
(\mu - A)
\\
&= 
\mu C_{\mu}
+
B C_{\mu}
-
C_{\mu} \mu
-
C_{\mu} A
\\
&= 
B C_{\mu} - C_{\mu} A.
\end{align*}
Thus,
$$
\mu^2
\|R(\mu, A) - R(\mu, B)\|
=
\mu^2
\left \|
B C_{\mu} - C_{\mu} A
\right \|.
$$
We will show that for a bounded linear operator $A$ in $\mathbb{X}$ the following is valid
\begin{equation}\label{7}
\lim_{\mu\to+\infty} R(\mu,A)=0.
\end{equation}
In fact, for sufficiently large $\mu$, say $\mu >\| A\|$, using the Neuman series, for large $\mu$, we have
\begin{align*}
\| R(\mu,A)\| 
= \frac{1}{\mu} \left\| R\left(1,\frac{1}{\mu}A\right) \right \|
&\le
\frac{1}{\mu} \left\| \sum_{n=0}^\infty \left( \frac{1}{\mu}A\right)^n \right\|  
\\
&\le
\frac{1}{\mu}
\frac{1}{1-  \frac{1}{\mu}\|A\|}.
\end{align*}
This proves \eqref{7}.
Next, by \eqref{7} and the identity
$(\mu-A)R(\mu,A)=I$
we have that
\begin{equation*}
\lim_{\mu\to+\infty} \mu R(\mu,A)=I,
\end{equation*}
so
\begin{equation*}
\lim_{\mu\to+\infty} \mu^2 C_\mu=I,
\end{equation*}
Finally, we have
$$
d_{Y}(A,B)
:=
\limsup_{\mu\to +\infty} 
\mu^2
\left \|
B C_{\mu} - C_{\mu} A\right \|
=
\|B-A\|,
$$
and this finishes the proof.

\item By a simple computation we have
\begin{equation}
\label{4}
R(\mu, A+C)-R(\mu, A) = R(\mu, A+C)CR(\mu,A). 
\end{equation}
It is known (see e.g. Pazy~\cite[Theorem 1.1, p.~76]{paz} that $A+C$ with $D(A+C)=D(A)$ generates a $C_0$-semigroup $(S(t))_{t\ge 0}$ satisfying
\begin{equation}\label{per}
\| S(t)\| \le Me^{(\omega +M\|C\|)t}, 
\end{equation}
so by the Hille-Yosida Theorem
$$
\| R(\mu, A)\| \le \frac{M}{\mu -\omega},\quad 
\| R(\mu,  A+C)\| \le \frac{M}{\mu -(\omega+M\| C\|)}, 
$$
for certain positive constants $M$ and $\omega$. Therefore,
\begin{align}
\limsup_{\mu \to \infty}  \mu^2 \| R(\mu, A)-R(\mu, A+C)\| 
& \le
\limsup_{\mu\to\infty } \frac{\mu^2M^2\| C\| }{(\mu -\omega)(\mu -(\omega +M\| C\|)} 
\nonumber
\\
&= M^2\| C\| <\infty. 
\label{bounded}
\end{align}

\item By Pazy~\cite[Theorem 2.1, p.~80]{paz} and the remark that follows it, there exists a positive constant $\delta$ such that if $0\le a \le \delta$, then,
$A+C$ generates an analytic semigroup $(S(t))_{t\ge 0}$ that satisfies
$$
\| S(t)\| \le M e^{(\omega +\Lambda (c))t},
$$
where $\lim_{c\to 0} \Lambda (c) =0$. Therefore, by \eqref{4}, as $A$ generates an analytic semigroup there are positive constants $K$ and $N$ (see Pazy~\cite[Theorem~5.5, p.~65]{paz}) such that if $\mu >N$, then
$$
\| AR(\mu,A)\| \le \frac{K}{\mu}.
$$
Hence,
\begin{align*}
d_{Y}(A, A+C)
&= 
\limsup_{\mu\to \infty} \mu^2 \| R(\mu,A+C)-R(\mu, A)\| 
\\
& \le
\limsup_{\mu\to\infty} 
\frac{\mu^2M}{\mu-\omega -\Lambda (c)} \left( a\| AR(\mu, A)\| +c\| R(\mu, A)\| \right)
\\
&\le
\limsup_{\mu\to\infty} \frac{\mu^2M}{\mu-\omega -\Lambda (c)} \left(  \frac{aK}{\mu} +\frac{cM}{\mu-\omega}\right)
\\
&\le aKM+cM^2.
\end{align*}
\end{enumerate}
The proof is completed.
\end{proof}
\begin{remark}
The above lemma has shown that the Yosida distance can measure the perturbation of a closed operator $A$ by an operator $C$ such that the domain of $A+C$ contains that of $A$. 
Below we will present an example to show that our concept of Yosida distance can actually be applied to larger classes of perturbation in which the requirement that $D(A) \subset D(A+C)$ is freed.
\end{remark}

\begin{example}\label{exa 1}
Consider equations of the form
\begin{equation}\label{Exa1}
\frac{dx(t)}{dt} = a x(t-1), \quad x(t)\in \mathbb{R},
\end{equation}
where $a\in \mathbb{R}$ is a given number. As shown in Hale~\cite{hal} this functional differential equation generates a $C_0$-semigroup $(T(t))_{t\ge 0}$  in the phase space $\mathbb{X}:= C([-1,0],\mathbb{R})$  with the generator $A_a$ defined as
$$
[A_a\phi ](t)=\begin{cases}\phi'(t) & \mbox{ if }  [-1,0),\\
a\phi (-1) & \mbox{ if } t=0,
\end{cases}
$$
with 
$$
D(A_a):= \left \{ \phi \in C^1([-1,0],\mathbb{R}) \mid \phi'(0) = a \phi (-1) \right \}.
$$
We are going to compute $R(\lambda, A_a):=(\lambda-A_a)^{-1}$ for sufficiently large $\lambda$. Given a function $\psi \in \mathbb{X}$, based on the definition of $A_a$ and by a simple computation we arrive at
\begin{equation}
\left( (\lambda I-A_a)^{-1} \psi \right)(t) = \left[ \frac{1}{\lambda -ae^{-\lambda}} \left( \lambda \int^0_{-1} e^{-\lambda s}\psi (s)ds +\psi (0)\right)  \right] e^{\lambda t} -\int^t_{-1} e^{\lambda (t-s)}\psi (s)ds.
\end{equation}
Therefore, by defining the operators $A_a$ and $A_b$ corresponding to two numbers $a$ and $b$ as above, for sufficiently large $\lambda>0$, we have the following
\begin{align}
& \lambda^2 \left | \left( R(\lambda , A_a)-R(\lambda , A_b) \psi\right) (t) \right |
\nonumber \\
&=\left| \left( \frac{\lambda^2}{\lambda -ae^{-\lambda}}-\frac{\lambda^2}{\lambda -be^{-\lambda}} \right) \left( \lambda \int^0_{-1} e^{-\lambda s}\psi (s)ds +\psi (0)\right) e^{\lambda t} \right| \nonumber \\
& = \frac{\lambda^2|a-b|}{(\lambda-ae^{-\lambda})(\lambda-be^{-\lambda})} \left| \lambda \int^0_{-1} e^{-\lambda s}\psi (s)ds +\psi (0)\right| \nonumber \nonumber \\
& \le \frac{\lambda^2|a-b|}{(\lambda-ae^{-\lambda})(\lambda-be^{-\lambda})} \left( (1-e^{-\lambda}) \sup_{-1\le s\le 0}|\psi (s)|+| \psi (0) |\right) \nonumber \nonumber \\
& \le \frac{2\lambda^2|a-b|}{(\lambda-ae^{-\lambda})(\lambda-be^{-\lambda})} \| \psi \| .
\end{align}
Finally, we arrive at
\begin{align}
d_Y(A_a,A_b) 
&= \lim_{\lambda \to +\infty } \sup_{\| \psi \| \le 1} \sup_{-1\le t\le 0} \lambda^2 \left | \left( R(\lambda , A_a)-R(\lambda , A_b) \psi\right) (t) \right | 
\nonumber \\
& \le 2|a-b|. 
\end{align}
\end{example}

\begin{claim}
Under the notations of Example \ref{exa 1}, for any reals $a$ and $b$, the following claims are true
\begin{enumerate}
\item The Yosida distance between $A_a$ and $A_b$ satisfies the estimate
\begin{align}
d_Y(A_a,A_b) 
& \le 2|a-b|. 
\end{align}
\item The inclusion 
$
D(A_a) \subset D(A_b)
$ 
is valid 
  if and only if $
a = b.$
\end{enumerate}
\end{claim}
\begin{proof} Part (1): It is clear from the above computation.\\
Part (2): By assumption, $\phi \in D(A_a)$ if and only if $\phi \in C^1([-1, 0], \mathbb{R})$ and
$$
\phi' (0) = a \phi (-1).
$$
This implies 
$$
\phi' (0) = a \phi (-1) = b \phi (-1).
$$
Take the function $\phi (t):= t+2$. Clearly, $\phi \in C^1([-1, 0], \mathbb{R})$ and $\phi'(0) = \phi (-1)=1$.
Hence $a = b$.
\end{proof}

\begin{remark}
This example shows that under even very small perturbation, the domain $D(A_a)$ of the generator $A_a$ of the $C_0$-semigroup associated with Eq. \eqref{Exa1} may change sharply, so the requirement that the domain $D(A_b)$ of the generator of  the $C_0$-semigroup associated with perturbed equation contain $D(A_a)$ does not make sense.
\end{remark}

The following theorem is the main result of this subsection on the roughness of exponential dichotomy under Yosida perturbation.
\begin{theorem}\label{the per}
Let $A$ be the generator of a $C_0$-semigroup that has an exponential dichotomy. Then, the $C_0$-semigroup generated by an operator $B$ also has an exponential dichotomy, provided that $d_{Y}(A,B)$ is sufficiently small.
\end{theorem}
\begin{proof}
First, we assume that both semigroups generated by $A$ and $B$ are contraction $C_0$-semigroups. Then, 
by Pazy~\cite[Lemma~3.4]{paz}, $e^{tA_\lambda }$ is a $C_0$-semigroup of contractions. 

Let $C$ and $D$ be two bounded linear operators in a Banach space $\mathbb{X}$. We will estimate the growth of $e^{tC}-e^{tD}$. By the Variation-of-Constants Formula that is applied to the equation $x'(t)=Cx(t)+(D-C)x(t)$ and by setting $x(t)=e^{tD}x$ we have
$$
x(t)=e^{tC}x+\int^t_0 e^{(t-s)C}(D-C)x(s)ds .
$$
For each $t\ge 0$, we have
\begin{align*}
\left \|e^{tC}x-e^{tD}x\right \|
&\le
\int^t_0 \left \| e^{(t-s)C}(D-C)e^{sD}x\right \| ds
\\
&\le
t\| C-D\|  \| e^{tC}\| \| e^{tD}\|  \| x\|.
\end{align*}
Therefore,
\begin{equation}
\| e^{tA_\lambda}-e^{tB_\lambda }\| \le  t \| A_\lambda-B_\lambda \| e^{tA_\lambda }\| \|e^{tB_\lambda }\|.
\end{equation}
Now we assume that $(T(t))_{t\ge 0}$ and $(S(t))_{t\ge 0}$ are the $C_0$-semigroups generated by $A$ and $B$ that satisfy 
$$
\| T(t)\| \le Me^{\omega t},
\quad
\| S(t)\| \le Me^{\omega t}
$$
for certain positive numbers $M$ and $\omega$. As is well known, for the Yosida approximation $A_\lambda$ of the generator $A$ of a $C_0$-semigroup $T(t)$ that satistifies $\|T(t)\| \le Me^{\omega t}$ the following estimate of the growth is valid (see e.g. Pazy~\cite[(5.25)]{paz})
$$
\| e^{tA_\lambda } \|
\le  Me^{2\omega t},
\quad
\| e^{tB_\lambda }\|
\le  Me^{2\omega t}.
$$
Hence, we have 
\begin{equation}\label{9}
\| e^{tA_\lambda}-e^{tB_\lambda }\|
\le 
t M^2\| A_\lambda-B_\lambda \| e^{4 t\omega }.
\end{equation}
By Pazy~\cite[Theorem 5.5]{paz}, for each $x\in \mathbb{X}$, we have
\begin{align}
\|T(t)x-S(t)x\| 
&=
\lim_{\lambda \to\infty} \| e^{tA_\lambda }x-e^{tB_\lambda }x\|  
\nonumber\\
& \le
tM^2e^{4\omega t} \limsup_{\lambda \to\infty } \| A_\lambda  - B_\lambda \|   
\nonumber\\
&=
tM^2e^{4\omega t} d_{Y}(A,B). \label{bounded2} 
\end{align}
Finally, if $d_{Y}(A,B)$ is sufficiently small, $\|T(1)-S(1)\|$ is sufficiently small as well, and thus, $(S(t))_{t\ge 0}$ has an exponential dichotomy.
\end{proof}
\begin{corollary}
Let $A$ be the generator of an exponentially dichotomous $C_0$-semigroup $(T(t))_{t\ge 0}$ in $\mathbb{X}$ and $C$ be a bounded linear operator in $\mathbb{X}$. Then, the operator $A+C$ with domain $D(A+C) =D(A)$ generates an exponentially dichotomous $C_0$-semigroup, provided that $\|C\|$ is sufficiently small.
\end{corollary}
\begin{proof}
Let $A$ generate a $C_0$-semigroup $(T(t))_{t\ge 0}$ that satisfies $\| T(t)\| \le Me^{\omega}$.  By \eqref{bounded} if $\|C\|$ is sufficiently small, then $d_{Y}(A,A+C)$ is sufficiently small as well, so by Theorem \ref{the per}, the semigroup generated by $A+C$ also has an exponential dichotomy.
\end{proof}

\begin{corollary}
Let $A$ be the generator of a hyperbolic analytic semigroup $(T(t))_{t\ge 0}$ in $\mathbb{X}$ satisfying $\| T(t)\| \le Me^{\omega t}$. Then, if $C$ is a $A$-bounded linear operator in $\mathbb{X}$, that is, it satisfies for all $x\in D(A)$
$$
\| Cx\| \le a\| Ax\| +c\| x\| 
$$
for certain positive constants $a$ and $c$. Then, $A+C$ generates an exponentially dichotomous analytic semigroup, provided that $a$ and $c$ are sufficiently small.
\end{corollary}
\begin{proof}
By assumption for each $\ x\in D(A)$ we have
\begin{align}
\| Cx\| & \le \gn{C}\cdot  \gn{x} \notag \\
 & \le \gn{C}\cdot \| A x\| +  \gn{C}\cdot \gn{x}.
\end{align}
As is well known (see e.g. Pazy~\cite[Theorem 2.1, p.~80]{paz}) for sufficiently small $\gn{C}$, the operator $A+C$ generates an analytic semigroup. Moreover, by \eqref{ana 4}
if  $\gn{C}$ is sufficiently small, $d_{Y}(A,A+C)$ is sufficiently small, so by Theorem~\ref{the per}, the semigroup generated by $A+C$ has an exponential dichotomy.
\end{proof}

\subsection{Invariant Manifolds}\label{Subsect-InvMa}
We consider evolution equations of the form
\begin{equation}\label{nonlinear eq}
u'(t) =Au(t),
\end{equation}
where $A$ is a nonlinear single-valued operator from $D(A)\subset \mathbb{X}$ to $\mathbb{X}$. We will assume that $A(0)=0$, $A$ is proto-differentiable in a neighborhood of $0$ (see the definition below) and the linearized evolution equation at $0$ (i.e. $u'=\partial A(0) u$) has an exponential dichotomy. Roughly speaking, our next result in this section to show that if the proto-derivative of $A$ is continuous at $0$ in the Yosida distance' sense, then there exist stable and unstable invariant manifolds in a neighborhood of $0$.

\subsubsection{Proto-Differentiability}
The convergence of sets in a complete metric space $(X,d)$ in this section is adapted from the similar concept from Rockafellar \cite{roc}. A family of sets $\{S_t\} _{t>0}$ in a complete metric space $(X,d)$ is said to \textit{converge} to a set $S\subset X$ as $t\downarrow 0$, written
\begin{equation}
S=\lim_{t\downarrow 0} S_t, 
\end{equation}
if $S$ is closed and
\begin{equation}\label{18}
\lim_{t\downarrow 0} 
\mathrm{dist} (w,S_t)
=
\mathrm{dist} (w,S),
\quad \mbox{ for all } w\in X,
\end{equation}
where ``$\mathrm{dist}$'' denotes the distance  $\mathrm{dist} (w,S):= \inf_{y\in S} \{d(w, y) \}$. It is often convenient to view \eqref{18} as
the equation
\begin{equation}
S=\liminf_{t\downarrow 0} S_t=\limsup_{t\downarrow 0} S_t, 
\end{equation}
where
\begin{equation}
\liminf_{t\downarrow 0} S_t 
=
\left \{
w\in \mathbb{X}
:
\limsup_{t\downarrow 0} \mathrm{dist}(w,S_t)=0
\right \}, 
\end{equation}
and
\begin{equation}
\limsup_{t\downarrow 0} S_t 
=
\left \{
w\in\mathbb{X}
:
\liminf_{t\downarrow 0} \mathrm{dist} (w,S_t)=0
\right \} .
\end{equation}
As in this paper all operators and functions under consideration are assumed to be sing-valued we will adapt the definition of proto-differentiability from Rockafellar~\cite{roc} accordingly.

\begin{assumption}\label{AsptD(G)}
Let $G: D(G)\subset V\subset \mathbb{X} \to \mathbb{X}$ be an operator.
We assume that domain $D(G)$ is an open subset of a vector subspace
$V \subset \mathbb{X}$.
\end{assumption}

\begin{definition}[Proto-differentiability]\label{def-protodiff}
Under Assumption \ref{AsptD(G)} let $x\in D(G)$ be a given vector.
For each $t \in (0, 1)$, let 
$T_t
\colon
V\subset \mathbb{X} \rightarrow \mathbb{X}$
be an operator defined as
\begin{equation}
T_t(w)=\frac{G(x+tw)-G(x)}{t}
\end{equation}
for each  $w \in B_{\varepsilon}(x)\cap V$,
where $\varepsilon$ is a sufficient small positive constant such that $\bar B_\varepsilon (x)\cap V  \subset D(G)$.
Then, we say that $G$ is \textit{proto-differentiable} at $u$ if there is a linear operator $T \colon D(T) \subset V \rightarrow \mathbb{X}$  
such that in $\bar B_\varepsilon (x) \times \mathbb{X}$ the graph of $T_t$ converges  to the graph of $T$ as $t \downarrow 0$. In this case we write $\partial G(x)=T$.
\end{definition}

\begin{remark}\label{rem proto}
As all functions considered in this paper are assumed to be single-valued, the proto-differentiablity of a function $G$ mentioned in Definition~\ref{def-protodiff} can be stated equivalently as follows (see \cite{aubeke,kat}): $G$ is proto-differentiable at $x\in D(G)$ if and only if $\partial_iG(x)=\partial _sG(x)$, where $\partial_iG(x)$ and $\partial _sG(x)$ are defined as:
\begin{enumerate}
\item The linear operator $\partial_i G(x)$ is defined at all $u\in \mathbb{X}$ and its value at $u$ is $v=:\partial_iG(x)u$  if for each sequence $\{ t_n\} \downarrow 0$, there exists a sequence $\{(u_n,v_n)\}\subset \mathbb{X}\times \mathbb{X}$ such that $(u_n,v_n)\to (u,v)$ in $\mathbb{X}\times \mathbb{X}$, $x+t_nu_n\in D(G)$ and
\begin{equation}
\frac{G(x+t_nu_n)-G(x)}{t_n}=v_n.
\end{equation}
\item
The linear operator $\partial_s G(x)$ is defined at all $u\in \mathbb{X}$ and its value at $u$ is $v=:\partial_sG(x)u$ if there exists a sequence $\{ t_n\} \downarrow 0$, and a sequence $\{(u_n,v_n)\}\subset \mathbb{X}\times \mathbb{X}$ such that $(u_n,v_n)\to (u,v)$ in $\mathbb{X}\times \mathbb{X}$, $x+t_nu_n\in D(G)$ and
\begin{equation}
\frac{G(x+t_nu_n)-G(x)}{t_n}=v_n.
\end{equation}
\end{enumerate}
By definition it is apparent that $\partial_iG(x)\subset \partial_sG(x)$.
\end{remark}

Before we proceed to studying the nonlinear perturbation of exponential dichotomy we consider some special cases.
\begin{example}\label{Ex-Aclosed}
Consider $G(x) = Ux$, where $U \colon D(U)\subset \mathbb{X} \to \mathbb{X}$ is a closed linear operator. Then, $G$ is proto-differentiable at every $x\in D(U)$ and $\partial G(x) = U$. 
\end{example}
\begin{proof}
Indeed, we have
$$
D_t(w)
=
\frac{G(u+tw)-G(u)}{t}
=
\frac{U(u+tw)-Uu}{t}
=
\frac{U(tw)}{t}
=
Uw,
$$
so, the operator $D_t$ and $U$ are identical in any neighborhood of $u$, and their graphs must be the same.
\end{proof}

\begin{definition}[G\^{a}teaux differentiability]
Suppose $\mathbb{X}$ and $\mathbb{X}$ are Banach spaces, $U\subseteq X$ is open, and $F: U\subset \mathbb{X} \rightarrow \mathbb{Y}$. Operator $F$ is \textit{G\^{a}teaux differentiable} at $u\in U$ if there exists an operator $L\in \mathcal{L}(\mathbb{X},\mathbb{X})$ such that
$$ 
\lim_{\tau \rightarrow 0}
\frac{F(u+\tau w) - F(u)}{\tau} =Lw,
$$
for all $w\in \mathbb{X}$. In this case we denote the G\^{a}teaux derivative of $F$ at $u$ by $dF(u)$.
\end{definition}

\begin{proposition}\label{M-pro gateau}
Let $G \colon D(G) = V \subset \mathbb{X} \rightarrow \mathbb{X}$ and let $V$ be equipped with the graph norm $\gn{x}:=\| x\| +\| Sx\|$, where $S:D(S)=V\subset \mathbb{X} \to \mathbb{X}$ is a closed linear operator. Then, the following assertions are true:
\begin{enumerate}
\item[(i)] If $G$ is a G\^{a}teaux differentiable operator from $(V, \gn{\cdot })$  to $(\mathbb{X},\| \cdot \|)$ and $T$ is its derivative at $x$. Then \begin{equation}
T \subset \partial_i G(x).
\end{equation}
\item[(ii)] If $G$ is a Fr\'{e}chet differentiable operator from $(V, \gn{\cdot })$  to $(\mathbb{X},\| \cdot \|)$ and $T$ is its derivative at $x\in V$ that is a closed operator in $\mathbb{X}$. Then, $G$ is proto-differentiable at $x\in V$ and 
\begin{equation}
 \partial G(x)=T.
\end{equation}
 Moreover, if $H=S+G$, then $\partial H(x)=S+G'(x)$, for all $x\in \mathbb{X}$.
\end{enumerate}
\end{proposition}
\begin{proof}
Part (i): By definition, as  $G$ is a G\^{a}teaux differentiable operator from $(V, \gn{\cdot })$  to $(\mathbb{X},\| \cdot \|)$ and $T$ is its  G\^{a}teaux derivative at $x$, we have
\begin{equation}
\lim_{t\to 0} \frac{\| G(x+tw)-G(x)-T(tw)\|}{\gn{tw}} =0,
\end{equation}
where $w\in B_\varepsilon (x)\cap V$ for a sufficiently small positive $\varepsilon$.
Therefore,
\begin{equation}
\label{4.11}
0=\lim_{t\to 0} \frac{\| G(x+tw)-G(x)-T(tw)\|}{\| tw\| +\|S(tw)\| } .
\end{equation}
We will prove that $T\subset \partial_i G(x) \subset \partial _s G(x) \subset T$. Let $0\not= w\in D(T)$ and  $\{ t_n\} \downarrow 0$ be any sequence. By \eqref{4.11},
\begin{equation}
\label{4.12}
0= \lim_{t\to 0} \frac{\| G(x+t_nw)-G(x)-T(t_nw)\|}{\| t_nw\| +\|S(t_nw)\| } .
\end{equation}
Set 
$$
u_n=w,\quad v_n :=\frac{G(x+t_nw)-G(x)}{t_n}.
$$
Then, for $n$ large enough $x+t_nw\in B_\varepsilon (x)\cap V\subset D(G)$, and setting $y:=G(x)$ we have
$$
y+t_nv_n=G(x)+t_n \frac{G(x+t_nw)-G(x)}{t_n} = G(x+t_nw),
$$
so, $(x+t_nu_n,y+t_nv_n)\in \mathrm{graph}(G)$. Obviously, $u_n\to w$. We will show that
\begin{equation}\label{4.13}
v_n = \frac{G(x+t_nw)-G(x)}{t_n}\to Tw,
\quad \mbox{ as } n\to \infty.
\end{equation}
By \eqref{4.12} 
\begin{equation}
\lim_{n\to\infty} \frac{\left \| \frac{G(x+t_nw)-G(x)}{t_n}-T(w)\right \|}{\| w\| +\|S(w)\| }.
\end{equation}
This shows that \eqref{4.13} is valid. Finally, we have that $(w,Tw) \in \mathrm{graph}(\partial_i G(x))$. By the arbitrary nature of $w$, this yields that
\begin{equation}\label{4.14}
T\subset \partial_i (G)(x) .
\end{equation}

\medskip
\noindent
Part (ii): 
Suppose that $(u,v)\in  \mathrm{graph}(\partial_s(G))$. This means that there exists a sequence $\{ t_n\} \downarrow 0$ and sequence $(u_n,v_n)\to (u,v)\in \mathbb{X}\times \mathbb{X}$ such that $G(x+t_nu_n )=G(x)+t_nv_n$ for each $n$. As $G$ is Fr\'{e}chet differentiable at $x$ we have
\begin{equation}
\lim_{n\to \infty} \frac{\| G(x+t_nu_n)-G(x)-T(t_nu_n)\|}{\gn{t_nu_n}} =0,
\end{equation}
so, as $u_n\to u\not=0$,
\begin{align}
0
&= \lim_{n\to\infty} \frac{\left \|  \frac{G(x+t_nu_n)-G(x)}{t_n}-T(u_n)\right \|}{\| u_n\| +\|S(u_n)\| }\\
&= \lim_{n\to\infty}\left \| \frac{G(x+t_nu_n)-G(x)}{t_n}-T(u_n)\right \| \\
&= \lim_{n\to\infty}\|  v_n-T(u_n)\| .
\end{align}
As $v_n\to v\in \mathbb{X}$, we have that $Tu_n\to v$ as $n\to \infty$. Since $T$ is a closed operator in $\mathbb{X}$ this yields that $u\in D(T)$ and $Tu=v$. In other words, $\partial_sG(x)\subset T$. Combined with Part (i) and the fact from the definition that  $\partial _iG(x)\subset \partial_sG(x)$ we have 
$$
T\subset \partial _iG(x)\subset \partial_sG(x)\subset T.
$$
This yields that $\partial G(x)$ exists and is equal to $T$, completing the proof.
\end{proof}

\begin{proposition}\label{pro-AFrechet}
Consider $x' = G(x) = Ux + F(x)$, where the linear operator $U$ is the generator of a $C_0$-semigroup in $\mathbb{X}$ and $F(\cdot )$ is Fr\'{e}chet differentiable at $x\in \mathbb{X}$.
Then, $G$ is proto-differentiable, and 
\begin{equation}
\partial G(x) = U + F'(x).
\end{equation}
\end{proposition}
\begin{proof}
We have
\begin{align*}
D_t(w)
=
\frac{G(u+tw)-G(u)}{t}
&=
\frac{U(u+tw) - Uu + F(u+tw)-F(u)}{t}
\\
&=
Uw
+
\frac{F(u+tw)-F(u)}{t}.
\end{align*}
Since $F(x)$ is Fr\'{e}chet differentiable, we have
$$
\lim_{t \rightarrow 0}
\left (
\frac{F(u+tw)-F(u)}{t}
-
F'(u)
\right )
= 0.
$$
Set $W:= U+F'(x)$. We will show that 
\begin{equation}\label{4-22}
W\subset \partial_iG(x)\subset \partial_s G(x)\subset W.
\end{equation}
First, we show that $W\subset \partial_iG(x)$.
Let $0\not= w\in D(W)$ and  $\{ t_n\} \downarrow 0$ be any sequence. By assumption,
\begin{align}
0&= \lim_{t\to 0} \frac{\| F(x+t_nw)-F(x)-F'(x)(t_nw)\|}{\| t_nw\| }\nonumber \\
&= \lim_{t\to 0} \frac{\| G(x+t_nw)-G(x)-W(t_nw)\|}{\| t_nw\| } \\
&= \lim_{n\to\infty} \frac{\left \|\frac{G(x+t_nw)-G(x)}{t_n}-W(w)\right \|}{\| w\|  }. \label{4.12-1}
\end{align}
Set 
$$
u_n=w,\quad v_n :=\frac{G(x+t_nw)-G(x)}{t_n}.
$$
Then, for $n$ large enough $x+t_nw\in B_\varepsilon (x)\cap V\subset D(G)$, and setting $y:=G(x)$ we have
$$
y+t_nv_n=G(x)+t_n \frac{G(x+t_nw)-G(x)}{t_n} = G(x+t_nw),
$$
so, $(x+t_nu_n,y+t_nv_n)\in \mathrm{graph}(G)$. Obviously, $u_n\to w$. We will show that
\begin{equation}\label{4.13-1}
v_n = \frac{G(x+t_nw)-G(x)}{t_n}\to Ww
\end{equation}
as $n\to \infty$. By \eqref{4.12-1}
\begin{equation}
\lim_{n\to\infty} \frac{\left \| \frac{G(x+t_nw)-G(x)}{t_n}-W(w)\right \|}{\| w\|  }.
\end{equation}
This shows that \eqref{4.13-1} is valid. Finally, this shows that $(w,Ww)\in \mathrm{graph}(\partial_i G(x))$. By the arbitrary nature of $w$, this yields that
\begin{equation}\label{4.14-1}
W\subset \partial_i (G)(x) .
\end{equation}
Next, we will show that $\partial_s G(x) \subset W$. Suppose that $(u,v)\in \mathrm{graph}(\partial_s(G))$. This means that there exists a sequence $\{ t_n\} \downarrow 0$ and sequence $(u_n,v_n)\to (u,v)\in \mathbb{X}\times \mathbb{X}$ such that $G(x+t_nu_n )=G(x)+t_nv_n$ for each $n$. As $F$ is Fr\'{e}chet differentiable at $x$ we have
\begin{align*}
0&= \lim_{t\to 0} \frac{\| F(x+t_nu_n)-F(x)-F'(x)(t_nu_n)\|}{\| t_nu_n\|} \\
&= \lim_{t\to 0} \frac{\| G(x+t_nu_n)-G(x)-W(t_nu_n)\|}{\| t_nu_n\|}\\
& =0,
\end{align*}
so, as $u_n\to u\not=0$,
\begin{align*}
0&= \lim_{n\to\infty} \frac{\left \|  \frac{G(x+t_nu_n)-G(x)}{t_n}-T(u_n)\right \|}{\| u_n\| +\|S(u_n)\| }\\
&= \lim_{n\to\infty}\left \| \frac{G(x+t_nu_n)-G(x)}{t_n}-W(u_n)\right \| \\
&= \lim_{n\to\infty}\|  v_n-W(u_n)\| .
\end{align*}
As $v_n\to v\in \mathbb{X}$, $Wu_n\to v$ as $n\to \infty$. Since $U$ is a closed operator in $\mathbb{X}$ and $F'(x)$ is a bounded operator, the operator $W$ is closed, so this yields that $u\in D(W)$ and $Wu=v$. In other words, $\partial_sG(x)\subset W$. Finally, \eqref{4-22} holds true,
completing the proof.
\end{proof}

\subsubsection{Families of Linear Operators Depending Continuously on a Parameter}
Recall that $d_{Y}(U,V)$ stands for the Yosida distance between two closed operators $U$ and $V$ in $\mathbb{X}$.
\begin{definition}\label{def cont}  
Let $A_\alpha$ be a family of (possibly unbounded) linear operators in $\mathbb{X}$ with parameter $\alpha \in S$, where $S$ is a metric space
satisfying 
$d_{Y}(A_{\alpha_0}, A_\beta )<\infty$ for every $\beta \in S$ and $\alpha_0$ is a certain element of $S$. The family of such operators $A_\alpha$ is said to be \textit{continuous} at $\alpha_0\in S$ if 
\begin{equation}
\lim_{\alpha \to \alpha_0}
d_{Y}(A_\alpha, A_{\alpha_0})
=
0.
\end{equation}
\end{definition}

\begin{example}
Let $A$ be the generator of a $C_0$-semigroup. Then, the family $A_C:=\{A+C, C\in \mathcal{L}(\mathbb{X})\}$ is continuous at $\alpha_0:=C=0$. In fact, as in Lemma~\ref{lem 1}, if the semigroup $T(t)$ generated by $A$ satisfies $\|T(t)\| \le Me^{\omega t}$, then
$$
d_{Y}(A,A+C)\le M^2 \|C\| .
$$
Therefore, when $C\to 0$ in $\mathcal{L}(\mathbb{X})$, $d_{Y}(A,A+C)\to 0$.
\end{example}

\begin{example}
Let $A$ be the generator of an analytic semigroup and $C$ be an $A$-bounded operator. Since $D(C)\supset D(A)$ and there are constants $a,\,c$ such that
$$
\| Cx\| \le a\| Ax\| +c\| x\| ,
$$
the restriction of $C$ on $D(A)$ is a bounded linear operator from $(D(A), \gn{\cdot})$ to $(\mathbb{X}, \|\cdot\| )$, where
$$
\gn{x}:=\|x\| +\|Ax\| , \ \quad x\in D(A).
$$
If we define $\mathcal{F}$ to be the family of all operators of the form $A_C:=A+C$ where $C$ is $A$-bounded, and the metric space $S$ as $\mathcal{L}((D(A),\gn{\cdot}),\mathbb{X})$, then
$$
\lim_{C\to 0} d_{Y}(A_C,A_0)=0.
$$
\end{example}

\subsubsection{Existence of Invariant Manifolds}
This subsection deals with the existence of invariant manifolds.

\begin{assumption}\label{M-assump124}
\
\begin{enumerate}[\quad \rm ({A2.}1)]
\item\label{AsptionA21}
There exists a real number $\omega$ such that $\omega I -A$ is m-accretive;
\item\label{AsptionA22}
The proto-derivative $\partial A(x)$ exists for each $x\in D(A)$ as a single-valued linear operator in $\mathbb{X}$ such that $\omega I-A$ is m-accretive.
\item\label{AsptionA23} 
Yosida distance $d_{Y}(\partial A(x),\partial A(0))$ satisfies
\begin{equation}
\sup_{x\in D(A)} d_{Y}(\partial A(x),\partial A(0))=\varepsilon <\infty .
\end{equation}
\end{enumerate}
\end{assumption}

The assumption yields that the operator $A$ generates a nonlinear semigroup in $\overline{D(A)}$ by the Crandall-Liggett Theorem (Theorem~\ref{the cra}). 

Let us consider the following family of equations associated with Eq. \eqref{nonlinear eq}
\begin{equation}\label{TM-Elam}
\left \{
\begin{aligned}
\frac{d}{dt} u_{\lambda}(t) 
& =
A^{\lambda} u_{\lambda} (t),\quad 0 \le t \le 1,
\\
u_{\lambda}(0) & = x,
\end{aligned}
\right .
\tag{$E_\lambda$}
\end{equation}
where 
\begin{equation}
J_\lambda^A:= (I-\lambda A)^{-1}, 
\quad
A^\lambda := \frac{1}{\lambda} \left( J^A_\lambda -I\right),
\end{equation}
for $\lambda >0$ and $\lambda \omega <1$. Note that with our notations if $U$ is a linear operator, then
$$
J^U_\lambda =(I-\lambda U)^{-1}= \frac{1}{\lambda }R\left (\frac{1}{\lambda}, U\right ).
$$
Therefore,
$$
U^\lambda = \frac{1}{\lambda} \left( \frac{1}{\lambda }R\left (\frac{1}{\lambda}, U\right ) -I\right)= U_{1/\lambda},
$$
where $U_\xi$ is the Yosida approximation of a linear operator $U$ with parameter $\xi$. Therefore, the Yosida distance between two linear operators $U$ and $V$ can be written as
\begin{align}
d_Y(U,V)&= \limsup_{\lambda \downarrow 0} \frac{1}{\lambda^2} \left \| R\left (\frac{1}{\lambda}, U\right )-R\left (\frac{1}{\lambda}, V\right )\right \|
\nonumber\\
&= \limsup_{\lambda \downarrow 0} \frac{1}{\lambda^2}\left \| \left (\frac{1}{\lambda} -U\right )^{-1} -\left (\frac{1}{\lambda} -V\right )^{-1} \right \|
\nonumber \\
&= \limsup_{\lambda \downarrow 0}  \frac{1}{\lambda} \left \| J^U_\lambda -J^V_\lambda \right \|
\label{Yo-1lambdaJUV}
\\
&=  \limsup_{\lambda \downarrow 0}  \frac{1}{\lambda} \left \| U_{1/\lambda}-V_{1/\lambda} \right \| . \label{m 4.36}
\end{align}
Recall that if $Q$ is a nonlinear operator such that $\omega I-Q$ is m-accretive if and only if 
\begin{equation}
\left \| J^Q_\lambda x-J^Q_\lambda y\right \| \le \frac{1}{1-\lambda \omega} \| x-y\|,
\quad x,\,y \in \mathbb{X}.
\end{equation}

By Kato~\cite[Lemmas 1.1, 1.2]{kat}, $A^\lambda$ is proto-differentiable in a neighborhood of $0$ and $\partial A^\lambda (x) =dA^\lambda (x)$ for $x$ in a small neighborhood of $0$. Consider the linearized equation for equation \eqref{TM-Elam} along a solution $S_\lambda (t)x_0$
\begin{equation}\label{TM-Llamw}
\left \{
\begin{aligned}
\frac{d}{dt} v_{\lambda}(t) 
+
d A^{\lambda} (S_{\lambda} (t)x_0) v_{\lambda} (t)
& = 0,
\quad 0 \le t \le 1,
\\
v_{\lambda}(0) 
& = w.
\end{aligned}
\right .
\tag{$L_{\lambda};w$}
\end{equation}
Since for each sufficiently small (but positive) $\lambda$,  $A^\lambda (x)$ is a Lipschitz operator with $A^\lambda (0)=0$, it generates a a semigroup $(S_\lambda (t))_{t\ge 0}$. Moreover, $u_\lambda (t):=S_\lambda (t)x$ is the classical solution of \eqref{TM-Elam}. By Kato~\cite[Proposition 3.1]{kat} the solution of the Cauchy problem \eqref{TM-Elam} satisfies
\begin{equation}\label{m v}
\| v_\lambda (t)\| \le e^{\mu t}\| w\| ,
\end{equation}
where $\mu =\omega /(1-\lambda \omega )$, $0<\lambda < 1/\omega.$

\begin{lemma}
Define $\mathcal{S}_{\lambda} \colon \mathbb{X} \rightarrow C([0,1], \mathbb{X})$ by
\begin{equation}
(\mathcal{S}_{\lambda}x)(t) = S_{\lambda}(t)x,
\quad \mbox{ for } t \in [0, 1].
\end{equation}
Then, we have
\begin{enumerate}
\item $\mathcal{S}_{\lambda}$ is G\^{a}teaux differentiable at each $x \in B_r(0)$ and the G\^{a}teaux derivative $d\mathcal{S}_{\lambda}(x)$ represents a unique solution of \eqref{TM-Llamw},
\begin{equation}
d\mathcal{S}_{\lambda}(x)w = v_{\lambda} (\cdot; 0, w)
\end{equation}
is solution of \eqref{TM-Llamw}.

\item Fix $w \in \mathbb{X}$. The mapping $z \mapsto d\mathcal{S}_{\lambda}(z)w$ is continuous from $B_r(0)$ into $C([0,1], \mathbb{X})$.

\item For $x,\, y \in B_r(0)$, 
\begin{equation}\label{Kato-f3.4}
\mathcal{S}_{\lambda}y - \mathcal{S}_{\lambda}x = \int_{0}^{1} d\mathcal{S}_{\lambda} (\theta y + (1-\theta)x)(y-x)d\theta
\end{equation}
in $C([0,1], \mathbb{X})$.
\end{enumerate}
\end{lemma}

\begin{proof}
See Kato~\cite[Proposition 3.2, Lemma 3.3, and Lemma 3.4]{kat}.
\end{proof}

\begin{lemma}\label{M-new}
Let $A(\cdot ),\, B(\cdot )$ be continuous functions from $[0,1]$ to $\mathcal{L}(\mathbb{X})$. Assume that $x(t)$ and $y(t)$ are the solutions of the Cauchy problems
\begin{align*}
& x'(t)=A(t)x(t), \quad x(0)= w,
\\
& y'(t)=B(t)y(t),\quad y(0)=w,
\end{align*}
respectively. Moreover, assume that the evolution operators  of these two equations satisfy 
$$
\| X(t,s)\| \le Me^{\omega (t-s)},\quad 
\| Y(t,s)\| \le Me^{\omega (t-s)},
\qquad
0 \le s \le t \le 1.
$$
Then,
\begin{equation}
\| x(t)-y(t)\| \le Me^{2\omega} \sup_{0\le \tau \le 1} \| A(\tau )-B(\tau)\| \cdot \| w\| .
\end{equation}
\end{lemma}
\begin{proof}
We have
$$
y'(t)=A(t)y(t)+(B(t)-A(t))y(t),
$$
so
$$
y(t)= X(t,0)w+\int^t_0 X(t,\tau ) (B(\tau )-A(\tau ))y(\tau )d\tau .
$$
Next,
\begin{align}
\| x(t)-y(t)\| & \le \int^t_0 \| X(t,\tau )\| \cdot \| B(\tau )-A(\tau )\| \cdot \|y(\tau )\| d\tau 
\notag\\
& \le Me^{2\omega} \sup_{0\le \tau \le 1} \| A(\tau )-B(\tau)\| \cdot \| w\| .
\end{align}
This completes the proof.
\end{proof}

\begin{corollary}\label{M-ss2d}
Under Assumption~\ref{M-assump124},
let $v_{\lambda}^{x}(t)$ and $v_{\lambda}^{0}(t)$ be solutions
to \eqref{TM-Llamw} with $x_0=x$ and $x_0=0$, respectively.
There exists $0 < r_0 < r$ such that 
$\|S_{\lambda}(t)x\| \le e^{2\omega t}\|x\| < r/2$, for small enough $\lambda$,
for all $x \in \mathbb{X}$
and $0 < t < 1$ 
the following estimate hold true
\begin{align}
\left \| v^x_\lambda (t) -v^0_\lambda (t)\right \|
\le 
e^{2\mu} \| w\|
\sup_{z \in \mathbb{X}}
d_Y(\partial A(z),\partial A (0)),
\end{align}
where $\mu =\omega /(1-\lambda \omega )$.
\end{corollary}
\begin{proof}
Applying the Lemma \ref{M-new} to 
\eqref{TM-Llamw}
where
$$
A(t) := d A^{\lambda} (S_{\lambda} (t)x),
\quad
B(t) := d A^{\lambda} (S_{\lambda} (t)0) = d A^{\lambda} (0)
$$
and
taking into account \eqref{m v}, we arrive at
\begin{align}
\left \| v^x_\lambda (t) -v^0_\lambda (t)\right \| & \le e^{2\mu} \| w\| \sup_{0\le t\le 1} \left \| dA^\lambda (S_\lambda (t)x) - dA^\lambda (S_\lambda (t)0) \right \| 
\notag\\
& \le e^{2\mu} \| w\|
\sup_{0\le t\le 1} 
\limsup_{\lambda \downarrow 0}
\left \|dA^\lambda (S_\lambda (t)x) - dA^\lambda (S_\lambda (t)0) \right \| 
\notag\\
& \le  e^{2\mu} \| w\|
\sup_{0\le t\le 1} d_Y(\partial A (S (t)x) ,\partial A (S (t)0) )
\notag\\
& \le  e^{2\mu} \| w\|
\sup_{z \in \mathbb{X}} d_Y(\partial A (z) ,\partial A (0) ). \label{M-4.47}
\end{align}
This finishes the proof.
\end{proof}

\begin{corollary}\label{M-Cophixy3omega}
Let $0$ be the stationary solution of \eqref{nonlinear eq} and let the Assumption~\ref{M-assump124}
be made. Let us denote by $(S(t))_{t \ge 0}$ and $(T(t))_{t \ge 0}$ the $C_0$-semigroups  generated by $A$ and $\partial A(0)$, respectively.
Then the following statements are true
\begin{equation}\label{M-phixphiy}
\|\phi (x) - \phi (y)\|
\le 
e^{3\omega} \| x - y\|
\sup_{z \in \mathbb{X}}
d_Y(\partial A(z),\partial A (0)),
\end{equation}
for all $x,\, y \in B_{r_0} (0)$ and $r_0$ is sufficiently small positive real number.
\end{corollary}
\begin{proof}
Set
\begin{align}
\phi(t)x 
& = S(t)x - T(t)x,
\label{TM-4}
\\
\phi_\lambda  (x) 
& = \mathcal{S}_\lambda (x) - \mathcal{T}_\lambda  (x),
\label{TM-5}
\end{align}
where
\begin{align*}
& 
(\mathcal{S}_\lambda x)(t) = S_\lambda  (t)x,
\quad \mbox{ for all } t \in [0,1],
\\
&
(\mathcal{T}_\lambda x)(t) = T_\lambda  (t)x, 
\quad \mbox{ for all } t \in [0,1].
\end{align*}

In order to prove that $(S(t))_{t \ge 0}$ and $(T(t))_{t \ge 0}$ are $\varepsilon_0$-close, it suffices to show that $\mathrm{Lip}(\phi_\lambda) \le \varepsilon$, where $\varepsilon = \varepsilon (\varepsilon_0)$ is positive and independent of $x$, $y$, i.e. the following estimate
\begin{equation}\label{TM-6}
\|\phi_\lambda  (x) - \phi_\lambda (y)\|
\le
\varepsilon
\|x-y\|,
\quad \mbox{ for all } x,\,y \in \mathbb{X}
\end{equation}
holds. Then, by letting $\lambda \downarrow 0$, for fixed $t$, we have
\begin{equation}\label{TM-7}
\|\phi (t) x - \phi (t) (y)\|
\le
\varepsilon
\|x-y\|,
\quad \mbox{ for all } x,\,y \in \mathbb{X},
\quad t \in [0, 1].
\end{equation}

Our task is now to prove \eqref{TM-6}. By Kato~\cite[Lemma~3.4]{kat}, we have
\begin{align}
& \|\phi_\lambda (x) - \phi_\lambda (y)\|
\notag \\
& = 
\left \|
	\mathcal{S}_{\lambda} (x)
	-
	\mathcal{S}_{\lambda} (y)
	-
	\mathcal{T}_{\lambda} (x)
	+
	\mathcal{T}_{\lambda} (y)
\right \|
\notag \\
& \le
\int_{0}^{1}
\left \|
	d
	\mathcal{S}_{\lambda}
	(\theta x + (1-\theta)y)(x-y)
	-
	d
	\mathcal{T}_{\lambda}
	(\theta x + (1-\theta)y)(x-y)
\right \|
d\theta.
\label{TM-8}
\end{align}
Let $\eta = \theta x+(1- \theta )y$. Then 
\begin{equation}
v_{\lambda}^{\eta } (t) = [d\mathcal{S}_{\lambda} (\eta)(x-y)](t)
\end{equation}
and 
\begin{equation}
v_{\lambda}^{0} (t):=v_{\lambda}(t) = [d\mathcal{T}_{\lambda} (\eta)(x-y)](t)
\end{equation}
are solutions to \eqref{TM-Elam} with the operators are $d A_{\lambda} (S_\lambda (t)\eta)$ and $dA_{\lambda} (0)$, respectively, and $w = x - y$. 
By Corollary~\ref{M-ss2d}, for $t \in [0,1]$, we have
\begin{align}
&
\left \|
	d
	\mathcal{S}_{\lambda}
	(\theta x + (1-\theta)y)(x-y)(t)
	-
	d
	\mathcal{T}_{\lambda}
	(\theta x + (1-\theta)y)(x-y)(t)
\right \|
\notag\\
& =
\left \|
	v_{\lambda}^{\eta} (t)
	-
	v_{\lambda} (t)
\right \|
\notag\\
& \le
e^{2\mu} \| x - y\|
\sup_{z\in \mathbb{X}}
d_Y(\partial A(z),\partial A (0)),
\end{align}
where $\mu =\omega /(1-\lambda \omega )$,
$\omega$ is a fixed positive number that makes $\omega I-A$ m-accretive. 
Therefore, 
\begin{align*}
&
\left \|
	d
	\mathcal{S}_{\lambda}
	(\theta x + (1-\theta)y)(x-y)
	-
	d
	\mathcal{T}_{\lambda}
	(\theta x + (1-\theta)y)(x-y)
\right \|
\\
& =
\sup_{t \in [0,1]}
\left \|
	d
	\mathcal{S}_{\lambda}
	(\theta x + (1-\theta)y)(x-y)(t)
	-
	d
	\mathcal{T}_{\lambda}
	(\theta x + (1-\theta)y)(x-y)(t)
\right \|
\\
& \le 
e^{2\mu} \| x - y\|
\sup_{z \in \mathbb{X}}
d_Y(\partial A(z),\partial A (0)).
\end{align*}
Next, we have
\begin{align*}
\|\phi_\lambda  (x) - \phi_\lambda (y)\|
& \le 
\int_{0}^{1}
e^{2\mu} \| x - y\|
\sup_{z \in \mathbb{X}}
d_Y(\partial A(z),\partial A (0))
d\theta
\\
& =
e^{2\mu} \| x - y\|
\sup_{z \in \mathbb{X}}
d_Y(\partial A(z),\partial A (0)).
\end{align*}
Letting $\lambda \downarrow 0$, we arrive at
\begin{equation}
\|\phi (x) - \phi (y)\|
\le
e^{3\omega} \| x - y\|
\sup_{z \in \mathbb{X}}
d_Y(\partial A(z),\partial A (0)).
\end{equation}
This completes the proof of \eqref{M-phixphiy}.
\end{proof}

\begin{definition}[Lipschitz invariant manifold]
Let $(S(t))_{t\ge 0}$ be a semigroup of (possibly nonlinear) operators on the Banach space $\mathbb{X}$. A set $\mathcal{M}\subset \mathbb{X}$ is said to be a {\it Lipschitz invariant manifold} for semigroup $(S(t))_{t\ge 0}$ if the phase space $\mathbb{X}$ is split into a direct sum $\mathbb{X}=\mathbb{X}^1\oplus \mathbb{X}^2$, where $\mathbb{X}^1$ and $\mathbb{X}^2$ are closed subspaces of $\mathbb{X}$, and there exists a Lipschitz continuous mapping $\Phi \colon \mathbb{X}^1\to\mathbb{X}^2$ so that
$\mathcal{M}=\mathrm{graph}(\Phi)$ and $S(t)\mathcal{M}\subset \mathcal{M}$ for $t \ge 0$.
\end{definition}

For brevity, a Lipschitz invariant manifold will be called simply invariant manifold if this does not cause any confusion.

\begin{definition}[{$\varepsilon$-close (see Minh-Wu \cite{minwu})}]
Two semigroups $(S(t))_{t \ge 0}$ and $(T(t))_{t \ge 0}$ on a Banach space $\mathbb{X}$ are said to be 
\textit{$\varepsilon$-close} 
if there exist a positive constant $\varepsilon$ such that
\begin{equation}\label{TM-3}
\|\phi(t)x - \phi(t)y\|
\le
\varepsilon
\|x-y\|,
\quad \mbox{ for all } t \in [0,1],
\quad 
x,\, y \in \mathbb{X},
\end{equation}
where
\begin{equation}\label{Q-phi}
\phi(t)x := S(t)x - T(t)x,
\quad 
\mbox{ for all }
x \in \mathbb{X}.
\end{equation}
\end{definition}

\begin{theorem}[Unstable invariant manifold]\label{TM-Theorem2.2}
Under Assumption~\ref{M-assump124}, let $0 \in \mathbb{X}$ be a stationary solution of Eq. \eqref{nonlinear eq}. Moreover, assume that the strongly continuous semigroup $(T(t))_{t \ge 0}$ has an exponential dichotomy with projection $P$. Then, there exists a positive constant $\varepsilon_0$, such that if $0 < \varepsilon < \varepsilon_0$, Eq. \eqref{nonlinear eq} has a unique invariant manifold $\mathcal{W}^{\rm u} \subset  \mathbb{X}$, presented as graph of a Lipschitz continuous mapping $\Phi \colon \mathrm{Ker} (P) \rightarrow \mathrm{Im} (P)$. Moreover, $\lim_{\varepsilon_0 \rightarrow 0} \mathrm{Lip}(\Phi) = 0$.
\end{theorem}
\begin{proof}
The theorem is an immediate consequence of Minh-Wu \cite[Lemmas 2.11, 2.12, 2.13]{minwu} and Corollary~\ref{M-Cophixy3omega}.
\end{proof}

\begin{theorem}[Stable invariant manifold]\label{TM-Theorem2.3}
With the assumptions in Theorem~\ref{TM-Theorem2.2}, the set
\begin{equation}\label{TM-10}
\mathcal{W}^{\rm s}
:=
\left \{
	x \in \mathbb{X}
	:
	\lim_{t \rightarrow +\infty}
	S(t)x = 0
\right \}
\end{equation}
is a stable invariant manifold of Eq. \eqref{nonlinear eq}, represented by the graph of a Lipschitz continuous mapping $\Psi \colon \mathrm{Im} (P) \rightarrow \mathrm{Ker} (P)$, i.e. $\mathcal{W}^{\rm s} = \mathrm{graph}(\Psi)$ and $S(t)\mathcal{W}^{\rm s} \subset \mathcal{W}^{\rm s}$, for all $t \ge 0$.
\end{theorem}
\begin{proof}
The theorem is an immediate consequence of Minh-Wu \cite[Theorem~2.16]{minwu} and Corollary~\ref{M-Cophixy3omega}.
\end{proof}

Below we present local versions of Theorems \ref{TM-Theorem2.2} and \ref{TM-Theorem2.3}. Recall that $B_r (0, \mathbb{Y})$ stands for the ball of radius $r$ centered at $0$ of $\mathbb{Y}$. Because we use $\mathbb{X}$ as the fixed phase space for Eq. \eqref{nonlinear eq}, for brevity, we denote $B_r (0, \mathbb{X})$ by $B_r (0)$ if this does not cause any confusion.

\begin{assumption}\label{M-assump123}
\
\begin{enumerate}[\quad \rm ({A3.}1)]
\item\label{AsptionA31} There exists a real number $\omega$ such that $\omega I -A$ is m-accretive;
\item\label{AsptionA32}  There exists a positive $r>0$ such that $A$ is proto-differentiable at every point $x\in B_r(0)\cap D(A)$ and
the proto-derivative $\partial A(x)$ is a single-valued linear operator in $\mathbb{X}$ such that $\omega I-A$ is m-accretive.
\item\label{AsptionA33}
Yosida distance $d_{Y}(\partial A(x),\partial A(0))$ satisfies
$d_Y(\partial A(x), \partial A(0))<\infty$ for all $x\in B_r(0)$ and it satisfies
\begin{equation}\label{AsptionA33dY}
\lim_{x\to 0} d_{Y}(\partial A(x),\partial A(0))=0.
\end{equation}
\end{enumerate}
\end{assumption}

This assumption yields that the operator $A$ is proto-differentiable in a neighborhood of $0$ and the proto-derivative $\partial A(\cdot)$ is continuous at $0$ in the Yosida distance's sense.

\begin{definition}[Local Lipschitz invariant manifold]
Let $(S(t))_{t\ge 0}$ be a semigroup of (possibly nonlinear) operators on the Banach space $\mathbb{X}$. A set $\mathcal{N}\subset \mathbb{X}$ is said to be a {\it local Lipschitz invariant manifold} for semigroup $(S(t))_{t\ge 0}$ around an equilibrium $0$ if 
$\mathbb{X}$ is split into a direct sum $\mathbb{X}=\mathbb{X}^1\oplus \mathbb{X}^2$, where $\mathbb{X}^1$ and $\mathbb{X}^2$ are closed subspaces of $\mathbb{X}$, and there exists a Lipschitz continuous mapping $\Phi \colon B_{r}(0, \mathbb{X}^1)\to \mathbb{X}^2$ and an open neighborhood $U$ of $0$ such that $\mathcal{N} \cap U = \mathrm{graph}(\Phi)$
and for each $t \ge 0$, $S(t)(\mathrm{graph}(\Phi)) \cap U \subset \mathrm{graph}(\Phi)$.
\end{definition}

\begin{theorem}[Local unstable manifold]\label{TM-Thm-USlocal}
Under the Assumption~\ref{M-assump123}, let $0 \in \mathbb{X}$ be a stationary solution of Eq. \eqref{nonlinear eq}. Moreover, assume that the strongly continuous semigroup $(T(t))_{t \ge 0}$ has an exponential dichotomy with projection $P$. Then, there exists a neighborhood $U$ of $0\in\mathbb{X}$,
such that Eq. \eqref{nonlinear eq} has a unique local invariant manifold $\mathcal{W}_{\rm loc}^{\rm u} \subset \mathbb{X}$, presented as
$\mathrm{graph}(\Phi) \cap U$, where $\Phi \colon \mathrm{Ker}(P) \rightarrow \mathrm{Im}(P)$ is a Lipschitz continuous mapping.
\end{theorem}
\begin{proof}
For the functions $\phi(\cdot)$ defined as in \eqref{Q-phi}, we consider the standard truncation procedure by defining
\begin{equation}
\phi_0(t)x:=
\begin{cases}
\phi(t)x & \mbox{ if }  \| x\| \le r_0,\\
\phi(t)(r_0x/\|x\|) & \mbox{ if } \| x\| >r_0.
\end{cases}
\end{equation}

It can be shown that $\phi_0(t)$ is Lipschitz continuous with Lipschitz coefficient
$\mathrm{Lip}(\phi_0(t)) = 2\mathrm{Lip}\left (\phi(t)\big|_{B_{r_0}(0)}\right )$
(see, for example,
Webb \cite[Proposition 3.10, p.95]{web85}).
Hence, by modifying \eqref{M-4.47} and \eqref{M-phixphiy}
we have
\begin{equation}\label{M-phixphiyloc}
\|\phi_0 (x) - \phi _0(y)\|
\le 
e^{3\omega} \| x - y\|
\sup_{\|z\| \le e^\omega r_0}
d_Y(\partial A(z),\partial A (0)),
\end{equation}
for all $x,\, y \in B_{r_0} (0)$ and $r_0$ is sufficiently small positive real number.
Next, we have
\begin{equation}
\mathrm{Lip}\left (\phi(t)\big|_{B_{r_0}(0)}\right )
\le 
e^{3\omega}
\sup_{\|z\| \le e^\omega r_0}
d_Y(\partial A(z),\partial A (0)).
\end{equation}
Since $d_Y(\partial A(z),\partial A (0)) \rightarrow 0$ as $r_0 \rightarrow 0$, hence $\mathrm{Lip}\left (\phi(t)\big|_{B_{r_0}(0)}\right ) \rightarrow 0$ as $r_0 \rightarrow 0$. Now, the assertions of Theorem \ref{TM-Theorem2.2} can be applied to $\phi_0(t)$. 
In fact, we choose $U = B_{r_0/2}(0, \mathbb{X}^2) \times B_{r_0/2}(0, \mathbb{X}^1) \subset B_{r_0} (0)$. By Theorem \ref{TM-Theorem2.2}, $\phi_0(t)$ has an invariant manifold $\mathcal{M}$, $\phi_0 (t)\mathcal{M} \subset \mathcal{M}$. Since $\mathcal{M} = \mathrm{graph}(\Phi)$, where $\Phi \colon \mathbb{X}^2 \rightarrow \mathbb{X}^1$, we have $\mathcal{N} := \mathrm{graph}\left (\Phi\big|_{B_{r_0/2}(0, \mathbb{X}^2)}\right ) \subset \mathcal{M}$.
This implies
$$
\phi_0 (t)\mathcal{N} \subset \phi_0 (t)\mathcal{M},
$$
that is,
$$
\phi_0 (t)\mathcal{N} \cap U \subset \phi_0 (t)\mathcal{M} \cap U \subset \mathcal{M} \cap U.
$$
Note that $\mathcal{M} \cap U = \mathrm{graph}\left (\Phi\big|_{B_{r_0/2}(0, \mathbb{X}^2)}\right )$, so
$$
\phi (t) \mathrm{graph}\left (\Phi\big|_{B_{r_0/2}(0, \mathbb{X}^2)}\right )
\subset 
\mathrm{graph}\left (\Phi\big|_{B_{r_0/2}(0, \mathbb{X}^2)}\right ).
$$
The proof is complete.
\end{proof}

\begin{theorem}[Local stable manifold]\label{TM-Thm-SMlocal}
Under the assumptions in Theorem~\ref{TM-Thm-USlocal}, there exists a neighborhood $U$ of $0\in\mathbb{X}$ such that the set
\begin{equation}\label{TM-10loc}
\mathcal{W}_{\rm loc}^{\rm s}
:=
\left \{
	x \in U
	:
	\lim_{t \rightarrow +\infty}
	S(t)x = 0
\right \}
\end{equation}
is a local invariant manifold of Eq. \eqref{nonlinear eq}, represented as
$\mathrm{graph}(\Psi) \cap U$, where $\Psi \colon \mathrm{Im} (P) \rightarrow \mathrm{Ker}(P)$ is a Lipschitz continuous mapping.
\end{theorem}
\begin{proof}
The proof is similar to that of Theorem~\ref{TM-Thm-USlocal} and so the details are omitted.
\end{proof}

\section{Applications and Examples}\label{Sect-AE}

\begin{example}[Semilinear equation]\label{exa 5.1}
Let $\mathbb{X}$ be a Banach space.
Consider the semilinear equation
\begin{equation}\label{TM-11}
\frac{du(t)}{dt} = (L + F)u(t),
\end{equation}
where $L \colon \mathbb{X} \rightarrow \mathbb{X}$ is the infinitesimal generator of a $C_0$-semigroup $(T(t))_{t\ge0}$ satisfying $\|T(t)\| \le  Me^{\omega t}$, where $\omega$ is a certain positive number, $F \colon \mathbb{X} \rightarrow \mathbb{X}$ is a nonlinear Fr\'{e}chet differentiable operator, and $\bar u = 0$ is the stationary solution of Eq. \eqref{TM-11}. Assume further that $F'$ is continuous in a neighborhood of $0$ and $F'(0)=0$.

Then, by Proposition \ref{pro-AFrechet}, the operator $A:=L+F$ is proto-differentable in a neighborhood of $0$. Moreover, $\partial A(x)=L+F'(x)$ and then, by Lemma~\ref{lem 1}
\begin{align}
d_Y(\partial A(x),\partial A(0))&=d_Y(L+F'(x),L) \notag \\
&\le \| F'(x)\| .
\end{align}
By a standard renorming  $|x|=\sup_{t\ge 0}\left \|e^{-\omega t}T(t)x\right \|$ we can reduce the problem to the case where $L$ generates a contraction semigroup. Next, we can use the standard truncation procedure by defining
\begin{equation}
F_0(x):=
\begin{cases}
F(x), & \mbox{ if }  \| x\| \le r_0,\\
F(r_0x/\|x\|), & \mbox{ if } \| x\| >r_0,
\end{cases}
\end{equation}
then, the function $F_0$ is globally Lipschitz. It is well known that in this case $-(L+F_0)$ is m-accretive (see e.g. \cite{ohatak, web76}). This process makes $A:=L+F_0$ satisfies all assumptions of Assumption \ref{M-assump124}.
\end{example}

\begin{example}[Semilinear equation]\label{exa 5.2}
Consider Eq. \eqref{TM-11} again with a different assumption that $L$ be the generator of an analytic $C_0$-semigroup $(T(t))_{t \ge 0}$ in $\mathbb{X}$ such that $\|T(t)\| \le e^{\omega t}$, $t\ge 0,$ and $\gn{x }:= \|x \|+\|Lx\|$ for all $x\in D(A)$, here we note that this norm $\gn{\cdot }$ makes $(D(A),\gn{\cdot })$ a Banach space. Assume further that $F:(D(A),\gn{\cdot})\to \mathbb{X}$ be a Fr\'{e}chet differentiable operator such that $F(0)=0$, $F'(0)=0$ and
$F'(\cdot)$
is countinous in a neighborhood of $0$. Then, by Proposition~\ref{M-pro gateau}, $F$ is proto-differentiable as a function from $D(A)\subset \mathbb{X}$ to $\mathbb{X}$ and $\partial F(x)= F'(x)$. Therefore, 
$$\partial (L+F)(x)=[L+F(x)]'= L+F'(x)$$ as $L$ is its Fr\'{e}chet derivative, itself.

Next, by Lemma \ref{lem 1},
\begin{align}
d_Y(\partial (L+F)(x),L) 
& = d_Y(L+F'(x),L)\notag\\
& \le K \gn{F'(x)},
\end{align}
where $K$ is a constant depending only on $L$ and for $U\in {\mathcal L}(V,\mathbb{X})$, $\gn{U}$ denotes the norm  of $U$. If we denote $A:=L+F$ in this case, then
\begin{align}
d_Y(\partial A(x),\partial A(0))&=d_Y(L+F'(x),L) \notag\\
&\le K\gn{ F'(x)},
\end{align}
where $K$ is a constant depending only on $L$. Therefore, if $F'(x)$ is continuous in $x$ around $0$, then, Assumption \ref{M-assump123} is made and Theorems \ref{TM-Thm-USlocal} and \ref{TM-Thm-SMlocal} apply.
\end{example}

\begin{example}\label{exa 5.3}
As a concrete example from PDE of Example \ref{exa 5.2} we can take the following: Let us consider the initial value problem
\begin{equation}
\left\{
\begin{aligned}
\frac{\partial u(t,x)}{\partial t} 
& = \frac{\partial^2u(t,x)}{\partial x^2}-au(t,x)+\sin \left( \frac{\partial^2 u(t,x)}{\partial x^2} \right),
&& t \ge 0,
&& x \in [0, \pi],
\\
u(t, 0) &= u(t, \pi) = 0,
&& t \ge 0,
\\
u(0,x) &= u_0 (x),
&&
&& x \in [0, \pi],
\end{aligned}
\right.
\end{equation}
where $a$ is a constant and $u_0(\cdot )\in L^2[0,\pi]$. If we set $\mathbb{X}:=L^2[0,\pi]$, then (see Pazy~\cite[Chapter 7]{paz}) or Lunardi~\cite{lun}), the operator 
$Ay:= y''-ay$ with domain $D(A)$ consisting of all $y\in \mathbb{X}$ such that $y'$ is absolutely continuous such that $y''\in \mathbb{X}$, $y(0)=y(\pi)$, will generate an analytic $C_0$-semigroup $(T(t))_{t\ge 0}$. Therefore, $(T(t))_{t\ge 0}$ has an exponential dichotomy if and only if the following equations
\begin{equation}
\lambda +a=-n^2, \quad n=1,2,\ldots 
\end{equation}
have no zero root (see Travis-Webb~\cite[p. 414]{traweb}). Since $F(y(\cdot )):= \sin (y(\cdot ))$ is continuously differentiable in $y(\cdot )\in D(A)$ all our assumptions in Theorems \ref{TM-Thm-USlocal} and \ref{TM-Thm-SMlocal} are made if $a\not =-n^2$ for any positive integer $n$.
\end{example}
%-----------------------
%-----------------------
\begin{example}[Age-dependent population dynamics]\label{exa age}
Let  $L^1:=L^1(0,\infty; \mathbb{R}^n)$ be a Bochner integrable function space, which norm is denoted by $\|\cdot\|_{L^1}$. Given two mappings $F \colon L^1 \rightarrow \mathbb{R}^n$ and $G \colon L^1 \rightarrow L^1$, we consider the following partial differential equation
\begin{equation}\label{Kato-P}
\left\{
\begin{aligned}
\frac{\partial u(t,a)}{\partial t} 
+
\frac{\partial u(t,a)}{\partial a} 
& =
G(u(t, \cdot))(a),
&& t \ge 0,
&& a \ge 0,
\\
u(t, 0) &= F(u(t, \cdot))
&& t \ge 0.
\end{aligned}
\right.
\end{equation}
For $i=1,\ldots,n$, define $K_i,\, J_i \colon L^1 \rightarrow [0, \infty)$ by 
$$
K_i \phi = \int_{0}^{\infty} k_i(a) \phi(a)da,
\quad
J_i \phi = \int_{0}^{\infty} j_i(a) \phi(a)da,
$$
respectively, where
$k_i,\, j_i \colon [0,\infty) \rightarrow \mathcal{L}(\mathbb{R}^n, [0,\infty))$
are given mappings. Then define $F \colon L^1 \rightarrow \mathbb{R}^n$ and $G \colon L^1 \rightarrow L^1$ by taking their $i$-th component as follows:
\begin{align*}
& F(\phi)_i = \int_{0}^{\infty} \beta_i (a, K_i(\phi)\phi_i(a)da
\quad \mbox{ for } \phi=(\phi_i)\in L^1,
\\
& G(\phi)_i (a)= -\mu(a,J_i\phi)\phi_i(a)da
\quad \mbox{ a.e. }
a>0
\quad \mbox{ for }
\phi=(\phi_i)\in L^1,
\end{align*}
where $\beta_1,\, \mu_i \colon [0,\infty) \times [0,\infty) \rightarrow [0,\infty)$ are given functions (see Webb \cite{web85} for details).

In the following, we assume that $F \colon L^1 \rightarrow \mathbb{R}^n$ and $G \colon L^1 \rightarrow L^1$ are continuously Fr\'{e}chet differentiable, i.e.,
\begin{itemize}
\item[\rm (F)] For any $\phi \in L^1$, there exists a $F'(\phi) \in \mathcal{L}(L^1, \mathbb{R}^n)$ such that
$$
F(\phi + h) = F(\phi) + F'(\phi)h + o_F(h), \quad h \in L^1,
$$
where $o_F \colon L^1 \rightarrow \mathbb{X}$, $\|o_F(h)\| \le b_F(r)\|h\|_{L^1}$ for $\|h\|_{L^1} \le r$, and $b_F \colon [0,\infty) \rightarrow [0,\infty)$ is a continuous increasing function satisfying $b_F(0)=0$; and there exists a continuous increasing function $d_F \colon [0,\infty) \rightarrow [0,\infty)$ such that
$$
\|F'(\phi) - F'(\psi)\|_{\mathcal{L}(L^1, \mathbb{X})} \le d_F(r)\|\phi - \psi\|_{L^1},
$$
for $\|\phi\|_{L^1} \le r,\, \|\psi\|_{L^1} \le r$.

\item[\rm (G)] For any $\phi \in L^1$, there exists a $G'(\phi) \in \mathcal{L}(L^1, L^1)$ such that
$$
G(\phi + h) = G(\phi) + G'(\phi)h + o_G(h), \quad h \in L^1,
$$
where $o_G \colon L^1 \rightarrow \mathbb{X}$, 
$\|o_G(h)\|_{L^1} \le b_G(r)\|h\|_{L^1}$
for $\|h\|_{L^1} \le r$, and $b_G: [0, \infty) \to [0, \infty)$ is a continuous increasing function satisfying $b_G(0)=0$; and there exists a continuous increasing function $d_G \colon [0,\infty) \rightarrow [0,\infty)$ such that
$$
\|G'(\phi) - G'(\psi)\|_{\mathcal{L}(L^1, L^1)} \le d_G(r)\|\phi - \psi\|_{L^1},
$$
for $\|\phi\|_{L^1} \le r,\, \|\psi\|_{L^1} \le r$.
\end{itemize}

Let $\bar{u}$ a be a stationary solution of \eqref{Kato-P}, i.e., $\bar{u} \in W^{1,1} = W^{1,1}(0, \infty; \mathbb{R}^n)$, $\bar{u} = F(\bar{u})$, and $\bar{u}' = G(\bar{u})$ where ``$'$'' stands for $d/da$ when the variable of functions in $W^{1,1}$ is represented by $a$. Fix $r_0 > 0$ such that $\|\bar{u}\|_{L^1} < r_0$. Then define the radial truncations $F_0$ and $G_0$ by
\begin{equation}
F_0(\phi) = 
\begin{cases}
F(\phi) & \mbox{ if } \|\phi\|_{L^1} \le r_0, \\
F(r_0\phi/\|\phi\|_{L^1}) & \mbox{ if } \|\phi\|_{L^1} > r_0,
\end{cases}
\end{equation}
and
\begin{equation}
G_0(\phi) = 
\begin{cases}
G(\phi) & \mbox{ if } \|\phi\|_{L^1} \le r_0, \\
G(r_0\phi/\|\phi\|_{L^1}) & \mbox{ if } \|\phi\|_{L^1} > r_0.
\end{cases}
\end{equation}
Then, the functions $F_0$ and $G_0$ are globally Lipschitz continuous and continuously Fr\'{e}chet differentiable on the ball $B_{r_0}$ in $L^1$. 

Now define an operator $A$ on $L^1$ by
\begin{equation}\label{def A}
A\phi = \phi' - G_0(\phi), \quad \mbox{ for } \phi \in D(A):=\left \{\phi\in W^{1,1} : \phi(0)=F_0(\phi)\right \}.
\end{equation}
Obviously, operator $A$ is not included in Examples \ref{exa 5.1} and \ref{exa 5.2} even though the formula defining $A$ in \eqref{def A} would suggests it is. This is due to the fact that the domain $D(A)$ is not the whole space $W^{1,1}$ (see \cite{hal, web85} for more information on the matter).

The properties of operator $A$ are summarized in the following proposition:
\begin{proposition}\label{kat-prs}
\
\begin{enumerate}
\item With $\omega = \|F_0\|_{\mathrm{Lip}} + \|G_0\|_{\mathrm{Lip}}$, $A+\omega I$ is a densely defined m-accretive operator in $L^1$.

\item With $\omega_u = \|F'(u)\|_{\mathcal{L}(L^1, \mathbb{X})} + \|G'(u)\|_{\mathcal{L}(L^1, L^1)}$, operator $\partial A(u) + \omega_u I$ is m-accretive in $L^1$.

\item For $u \in D(A) \cap B_r(\bar{u})$, $\partial A(u)$ exists and
\begin{equation}
\mathrm{graph}(\partial A(u)) = \lim_{t \downarrow 0} t^{-1}[\mathrm{graph}(A)-(u, Au)].
\end{equation}

\item Operator $\partial A(u) + \omega I$ is m-accretive in $L^1$ for $u \in D(A) \cap B_r(\bar{u})$.

\item There exist $\lambda_{\bar{u}}$ and a  nondecreasing $L_{\bar{u}} \colon [0,\infty) \rightarrow [0,\infty)$ such that
\begin{equation}\label{kat-p56}
\left \|J_{\lambda}^{\partial A(z)}v - J_{\lambda}^{\partial A(u)}v\right \|_{L^1} \le \lambda \|z-u\|_{L^1} L_{\bar{u}}(\|v\|_{L^1})
\end{equation}
for $0 < \lambda < \lambda_{\bar{u}}$, $z,\,u \in B_{\delta_{\bar{u}}}(\bar{u}) \cap D(A)$ and $v \in L^1$.
\end{enumerate}
\end{proposition}
\begin{proof}
See Kato~\cite[Propositions 5.2, 5.4, 5.5 and 5.6]{kat}.
\end{proof}
By \eqref{Yo-1lambdaJUV} and \eqref{kat-p56} we have
\begin{equation}\label{kat-p56eq}
d_Y (\partial A(z), \partial A(0) )=\frac{1}{\lambda} \left \|J_{\lambda}^{\partial A(z)}v - J_{\lambda}^{\partial A(u)}v\right \|_{L^1} \le \|z-u\|_{L^1} L_{\bar{u}}(\|v\|_{L^1}).
\end{equation}
This implies that
\begin{equation}
\lim_{z \rightarrow 0}
d_Y (\partial A(z), \partial A(0) = 0.
\end{equation}
This means that the condition \eqref{AsptionA33dY} is fulfilled. Therefore, all conditions listed in  Assumption~\ref{M-assump123} are satisfied. Applying Theorems \ref{TM-Thm-USlocal} and \ref{TM-Thm-SMlocal}, the age-dependent population model has a local stable, unstable manifold near $\bar u$ if the linear system $x'(t)=-\partial A(\bar u)x$ has an exponential dichotomy.
\end{example}

\bibliographystyle{amsplain}

\end{document}